\documentclass[11pt,oneside]{amsart}
\usepackage{amsmath,ifthen, amsfonts, amssymb, srcltx,
   amsopn, textcomp}

\usepackage{hyperref}
\usepackage{graphicx}

\usepackage[small]{caption}

\usepackage{tikz}
\usetikzlibrary{matrix}
\usetikzlibrary{decorations.pathreplacing}

\usepackage{geometry}
\usepackage{graphics}

\usepackage{mathrsfs}
\usepackage{stmaryrd}
\usepackage{tikz}
\usepackage{tikz-cd}
\usepackage{framed}
\usepackage{mathtools}
\usepackage{xfrac}
\usepackage{faktor}
\usepackage{stackengine}
\usepackage{braket}
\usepackage{comment}
\usepackage{cancel}
\usepackage{xcolor}

\usepackage{mathrsfs}

\usepackage{amssymb}

\tikzset{
math to/.tip={Glyph[glyph math command=rightarrow]},
loop/.tip={Glyph[glyph math command=looparrowleft, swap]},
loop'/.tip={Glyph[glyph math command=looparrowleft]},
 weird/.tip={Glyph[glyph math command=Rrightarrow, glyph length=1.5ex]},
  pi/.tip={Glyph[glyph math command=pi, glyph length=1.5ex, glyph axis=0pt]},
}

\usepackage{stackengine}

\newcommand{\showcomments}{yes}
\renewcommand{\showcomments}{no}

\newcommand{\hidetodo}[1]
{\ifthenelse{\equal{\showcomments}{yes}}%
{#1}
}

\newsavebox{\commentbox}
\newenvironment{com}%
{\ifthenelse{\equal{\showcomments}{yes}}%
{\footnotemark
        \begin{lrbox}{\commentbox}
        \begin{minipage}[t]{1.25in}\raggedright\sffamily\tiny
        \footnotemark[\arabic{footnote}]}
{\begin{lrbox}{\commentbox}}}%
{\ifthenelse{\equal{\showcomments}{yes}}%
{\end{minipage}\end{lrbox}\marginpar{\usebox{\commentbox}}}
{\end{lrbox}}}

\newtheorem{thm}{Theorem}[section]
\newtheorem{lem}[thm]{Lemma}

\newtheorem{conj}[thm]{Conjecture}

\newtheorem{conje}[thm]{Conjecture}

\theoremstyle{definition}
\newtheorem{defn}[thm]{Definition}
\newtheorem{rem}[thm]{Remark}

\newtheorem{prob}[thm]{Problem}

\DeclareMathOperator{\image}{image}

\DeclareMathOperator{\rim}{Rim}
\DeclareMathOperator{\length}{Length}

\DeclareMathOperator{\rank}{rank}

\DeclareMathOperator{\BS}{BS}
\DeclareMathOperator{\diam}{Diam}
\DeclareMathOperator{\dist}{Dist}
\DeclareMathOperator{\height}{Height}

\newcommand{\semidirect}{\ensuremath{\rtimes}}

\newcommand{\field}[1]{\mathbb{#1}}
\newcommand{\integers}{\ensuremath{\field{Z}}}

\newcommand{\euler}{\chi}

\begin{document}

\title[Negative Immersions and Finite Height Mappings]{Negative Immersions and Finite Height Mappings}

\author[B.~Abdenbi]{Brahim Abdenbi}
\email{brahim.abdenbi@mail.mcgill.ca  \, \, \, wise@math.mcgill.ca}
\author[D.~T.~Wise]{Daniel T. Wise}
          \address{Dept. of Math. \& Stats.\\
                    McGill Univ. \\
                    Montreal, Quebec, Canada H3A 0B9}
\thanks{Research supported by NSERC}
\date{\today}

\maketitle

\begin{com}
{\bf \normalsize COMMENTS\\}
ARE\\
SHOWING!\\
\end{com}

\begin{abstract}
Given a monomorphism $\Psi:\mathcal{H}\rightarrow \mathcal{F}$ where $\mathcal{H}$ is a proper free factor of the free group $\mathcal{F}$, we show the associated mapping torus $X$ of $\Psi$ has negative immersions iff $\mathcal{H}$ has finite height in $\pi_1X$ iff $\Psi$ is fully irreducible. We survey related properties and discuss possible directions to pursue further.
 \end{abstract}

\section{Introduction}

 Free-by-cyclic groups have been of continual interest in combinatorial and geometric group theory especially because of works of G.~Baumslag \cite{MR316570,MR1243634}, Bestvina-Feighn-Handel 
\cite{BestvinaHandel92, BestvinaFeighnHandel2000}, and have an increasing interest because of virtual algebraic fibering \cite{Kielak2020}. Among their most notable properties is that they 
are \emph{coherent} which means that every finitely generated subgroup is finitely presented \cite{FeighnHandelCoherence}. A free-by-cyclic group $\mathcal{G}=\mathcal{F}\semidirect_\Psi \integers$ is often studied by thinking of $\mathcal{F}$ as $\pi_1F$ for some based graph $F$, and representing $\Psi$ by a basepoint preserving map $F\rightarrow F$, and regarding $\mathcal{G}$ as  $\pi_1X$ where $X$ is the mapping torus corresponding to $\Psi$. 

The study of coherent groups led to the following notion:
a $2$-complex $X$ has \textit{nonpositive immersions} if for any combinatorial immersion $Y\rightarrow X$, with $Y$ compact, connected, and collapsed (meaning $Y$ has no free faces), either $\pi_1Y$ is trivial or $\chi\left(Y\right)\leq 0$. A  straightforward proof was given in \cite{WiseNonPositiveCoherence} showing that when  $X$ is the mapping torus of a $\pi_1$-injective map $\psi:F\rightarrow F$ of a graph, then $X$ has nonpositive immersions. Recent progress has been made on the conjecture that nonpositive immersions implies coherence of $\pi_1X$ \cite{Jaikin-Linton23}.

The following stricter form of nonpositive immersions was suggested in \cite{WiseSectional02}:
A $2$-complex $X$ has \textit{negative immersions} if there exists $c>0$ such that for any combinatorial immersion $Y\rightarrow X$, where $Y$ is compact, connected, collapsed, and with no isolated edges, either $\pi_1Y$ is trivial or $\chi\left(Y\right)\leq -c|Y|_2$ where $|Y|_2$ is the number of $2$-cells in $Y$.
This stronger form permits a geometric proof of coherence, and also provides a proof of local-quasiconvexity when $X$ is a negatively curved 2-complex, (and conjecturally in general).
We refer to \cite{WiseSectional02, LouderWiltonNegative2018, WiseNonPositiveCoherence, MR4386825} for more on negative immersions.

In view of the nonpositive immersion property for the mapping torus $X=M(\psi)$ of a $\pi_1$-injective endomorphism $\psi:F\rightarrow F$,
we were motivated to explore the negative immersion property for a ``partial mapping torus''
$X=M(\psi)$ where $F$ is a finite graph, and where $\psi:H\rightarrow F$ is a $\pi_1$-injective map from a subgraph $H\subset F$ to $F$.

Such partial mapping tori arise explicitly in the work of Feighn-Handel. Indeed,
they show that all f.g.\ subgroups of a Free-by-Cyclic group (or indeed, ascending HNN extension of a free group) arise as $\pi_1X$ where $X=M(\psi)$ is a partial mapping torus.

In this case $\mathcal{G}=\pi_1X$ is an HNN extension $\mathcal{F}*_{\mathcal{H}^t=\mathcal{K}}$, where $\mathcal{F}=\pi_1F$ and $\mathcal{H}=\pi_1H$. Note that when $H=F$ one obtains an ascending HNN extension - and a (f.g.\ Free)-by-Cyclic group when $\psi$ is $\pi_1$-surjective.
But our interest is the case where $H$ is a proper free factor, in which case $\mathcal{G}$ is deficient and indeed $\euler(\mathcal{G})<0$.

We began our investigation by imposing a high level of ``malnormality'' on the partial endomorphism and proved:
\begin{thm}\label{ncor:nthm1}Let $F$ be a finite graph and let $H\subset F$ be a subgraph. Let $X$ be the mapping torus of a cellular immersion $\psi:H\rightarrow F$.  Suppose $\psi^{-1}\left(H\right)$ is homeomorphic to a forest. Then $X$ has negative immersions.
\end{thm}
The hypothesis of Theorem~\ref{ncor:nthm1} has a malnormal flavor insofar as it asserts that $\psi_*(\mathcal{H})$ has trivial intersection with conjugates of $\mathcal{H}$.
Theorem~\ref{ncor:nthm1} is proven in the text as Theorem~\ref{thm:malnormalhnn}.
Its generalizations have somewhat more challenging proofs, but more striking statements which we now turn to:

We first state the generalization parallel to the statement of Theorem~\ref{ncor:nthm1}:
\begin{thm}\label{thm: preimages}
Let $F$ be a finite graph and let $H\subset F$ be a subgraph. Let $X$ be the mapping torus of a cellular immersion $\psi:H\rightarrow F$.  
Then $X$ has negative immersions if and only if
 $\psi^{-n}\left(H\right)$ is a forest for some $n\geq 0$.
 \end{thm}
Here $\psi^{-i}=\left(\psi^{i}\right)^{-1}$ whenever the \textit{partial composition} $\psi^{i}$ is defined. This led us to the notion of ``finite directed height'' (see  Definition~\ref{defn:composition} and Definition~\ref{defn:combinatorial}) which turns out to be equivalent to finite height \cite{GMRS98}, which we now recall:

 Given a subgroup $\mathcal{H}$ of $\mathcal{G}$, the \textit{height} of $\mathcal{H}$ in $\mathcal{G}$, denoted by  $\height\left(\mathcal{H}\right)$, is the supremal number of distinct cosets $\left\{\mathcal{H}g_i\right\}_{i\in I}$ such that $\displaystyle\bigcap_{i\in I} \mathcal{H}^{g_i}$ is infinite. \textcolor{black}{In particular, the height of $\mathcal{H}$ in $\mathcal{G}$ is $1$ precisely when $\mathcal{H}$ is a malnormal subgroup.}

Our main result about partial mapping tori relates height and negative immersions as follows:
\begin{thm}\label{thm:main1thm}
    Let $\Psi: \mathcal{H}\rightarrow \mathcal{F}$ be a monomorphism, where $\mathcal{H}$ is a proper free factor of a free group $\mathcal{F}$. Then the mapping torus $X$ of $\Psi$ has negative immersions if and only if $\height\left(\mathcal{H}\right)<\infty$ in $\pi_1X$.
\end{thm}

$\height(H)=\infty$ corresponds to nontrivial intersections of the form $H^{t^n}\cap g_nHg_n^{-1}$
for each $n\geq 1$ and some element $g_n\in F$.
Hence arbitrarily long directed segments in the Bass-Serre tree with infinite stabilizer. However, a much stronger failure of finite height occurs, in the form of a self-conjugation.

Thus negative immersions can be re-interpreted in an attractive way using a natural generalization of fully irreducible partial endomorphisms. A partial endomorphism $\Psi: \mathcal{H}\rightarrow \mathcal{F}$ is \textit{fully irreducible} if there does not exist $n>0$, a proper free factor $\mathcal{H}'\subset \mathcal{H}$, and $g\in \mathcal{F}$ such that $\Psi^n\left(\mathcal{H}'\right)\subset g^{-1}\mathcal{H}'g$. 
The standard notion of fully irreducible endomorphism focuses on the case where $\mathcal{H}=\mathcal{F}$ \cite{MR1147956}. Using this language, we show the following statement proved in the text as Theorem~\ref{thm:irreducible is finite height}:
\begin{thm}\label{thm:irreducible is finite height}
Let $\mathcal{H}$ be a proper free factor of a finitely generated free group $\mathcal{F}$, and let $\Psi:\mathcal{H}\rightarrow \mathcal{F}$ be a monomorphism. Let $X$ be the standard $2$-complex of the HNN extension of $\mathcal{F}$ with respect to $\Psi$. Then $X$ has negative immersions if and only if $\Psi$ is fully irreducible.  
\end{thm}


\begin{prob}\label{Ask Mahan}
Let $\mathcal{H}$ be a proper free factor of a free group $\mathcal{F}$,
and let $\Psi:\mathcal{F}\rightarrow \mathcal{F}$ be a fully irreducible endomorphism. 
Is $\mathcal{H}$ quasiconvex in $\mathcal{G} = \mathcal{F}\semidirect_\Psi \integers$? 
Is every  non-quasiconvex finitely generated subgroup of $\mathcal{G}$ a virtual (generalized) fiber? Here ``generalized'' allows for a free subgroup that is conjugated properly into itself.
\end{prob}
An affirmative answer to Problem~\ref{Ask Mahan} would show that our groups are locally quasiconvex.

It is expected that when $X$ has negative immersions,
$\pi_1X$ is a locally quasiconvex hyperbolic group.
There is now some evidence for the following conjecture which we hope will inform the next entries into this topic:
\begin{conj}\label{conj:local quasiconvex characterization}
    Let $\mathcal{G}$ be a locally quasiconvex hyperbolic group.
    Then $\mathcal{G}$ has a finite index subgroup $\mathcal{G}'$ such that
    $\mathcal{G}'=\pi_1X$ where $X$ is a mapping torus of a fully irreducible partial endomorphism of a free group.
\end{conj}

One direction of Conjecture~\ref{conj:local quasiconvex characterization} could be approached by proving virtual specialness of locally quasiconvex (locally indicable) hyperbolic groups - itself an interesting conjecture. And combining it with some of the current trend of using vanishing $L^2$-Betti number to obtain virtual fibering \cite{Kielak-Linton23}.
Another direction is equivalent to showing that the subgroup $\mathcal{H}$
has the \emph{finitely generated intersection property} in the sense that $\mathcal{H}\cap \mathcal{K}$ is finitely generated whenever $\mathcal{K}$ is finitely generated.

In Section~\ref{sec:height background}, we give background. In Section~\ref{sec:malnormalcoherence} we prove a special case of the main theorem relating malnormality and negative immersions. In Section~\ref{sec:finiteheight}, we prove the main theorems, and in Section~\ref{sec:generalization to pi1 mappings} we generalize our results to $\pi_1$-injective maps. Finally, in Section~\ref{sec:related properties} we discuss related properties.



\section{Background}\label{sec:height background}
We work in the category of $CW$-complexes. Let $Y$ be a $CW$-complex. We denote by $Y^{k}$ the $k$-skeleton of $Y$ and by $|Y|_k$ the number of $k$-cells in $Y$. Given complexes $X$ and $Y$, a map $Y\rightarrow X$ is \textit{cellular} if it maps $Y^k$ into $X^k$ for all $k$. It is \textit{combinatorial} if it maps open cells of $Y$ homeomorphically onto open cells of $X$. It is an \textit{immersion} if it is locally injective. A complex is collapsed if it has no free faces. A $1$-cell (edge) is \textit{isolated} if it is not a face of a $2$-cell. A $2$-complex $X$ has \textit{negative immersions} if there is $c>0$ such that for any combinatorial immersion $Y\rightarrow X$ with $Y$ compact, connected, collapsed (\textcolor{black}{containing no free faces}), and containing no isolated edges, either $\pi_1Y$ is trivial or $\chi\left(Y\right)\leq -c|Y|_2$ where $|Y|_2$ is the number of $2$-cells in $Y$ and $\chi \left(Y\right)$ is the Euler characteristic of $Y$.

A group $\mathcal{G}$ is \textit{coherent} if every finitely generated subgroup of $\mathcal{G}$ is finitely presented. The proof of Theorem~\ref{thm:negative Imm give coherence} can be found in \cite {WiseCoherenceSurvey}.
\begin{thm}\label{thm:negative Imm give coherence}
Let $X$ be a compact $2$-complex with negative immersions. Then $\pi_1X$ is coherent.
\end{thm}

A \textit{graph} $F$ is a $1$-dimensional $CW$-complex whose \textit{vertices} and \textit{edges} are the $0$-cells and $1$-cells, respectively. There exist two \textit{incidence} maps $\tau_1, \tau_2 : F^1\rightarrow F^0$ mapping each edge $e\in F^1$ to its \textit{boundary vertices}, $\tau_1\left(e\right),\ \tau_2\left(e\right)$ called initial and terminal vertex, respectively. Each edge is \textit{oriented} from its initial vertex to its terminal vertex. The \textit{degree} of a vertex $v$ relative to the graph $F$, denoted by $\deg_F\left(v\right)$, is the number of edges in $F^1$ containing $v$ as an initial or terminal vertex. An edge whose initial and terminal vertices coincide with $v$ counts twice in $\deg_F\left(v\right)$. A \textit{leaf} is a vertex of degree $1$ and a \textit{spur} is an edge containing a leaf. 
A graph is \textit{trivial} if it is a union of vertices. A \textit{tree} is a non-empty graph  \textcolor{black}{with no embedded circles} and a \textit{forest} is a disjoint union of trees. The \textit{empty graph} is the graph with no edges and no vertices. We consider the empty graph as a forest.

A \textit{graph of graphs} $X$ with underlying graph $\Gamma_X$, \textit{vertex-spaces} $\left\{X_v\right\}_{v\in {\Gamma_X^0}}$, and \textit{edge-spaces} $\left\{X_e\right\}_{e\in {\Gamma_X^1}}$ is a topological space $X$ obtained as a quotient of graphs $\left\{X_v\right\}_{v\in {\Gamma_X^0}}$ and $\left\{X_e\times I\right\}_{e\in {\Gamma_X^1}}$ in the following manner: for each edge $e\in {\Gamma_X^1}$ with boundary vertices $v_1=\tau_1\left(e\right), v_2=\tau_2\left(e\right)$, the edge-space $X_e\times I$ is attached to the vertex-spaces $X_{v_1}, X_{v_2}$ via an \textit{outgoing} attaching map $X_e\times \left\{0\right\}\rightarrow X_{v_1} $ and an \textit{incoming} attaching map $X_e\times\left\{1\right\}\rightarrow X_{v_2}$. The Euler characteristic of the resulting space is given by
$$\chi\left(X\right)=\displaystyle\sum_{v\in \Gamma_X^0}\chi\left(X_v\right)-\displaystyle\sum_{e\in \Gamma_X^1}\chi\left(X_e\right)$$

A subgroup $\mathcal{H}\subset \mathcal{G}$ is \textit{malnormal} if $g\mathcal{H}g^{-1}\cap \mathcal{H}=1_\mathcal{G}$ whenever $g\notin \mathcal{H}$. The pair $\mathcal{H}, \mathcal{K}\subset \mathcal{G}$ is malnormal if $g\mathcal{H}g^{-1}\cap \mathcal{K}=1_\mathcal{G}$ for all $g\in \mathcal{G}$. 
\section{Malnormality and Negative Immersions }\label{sec:malnormalcoherence}
\begin{defn}\label{defn:boundary}
Let $H$ be a subgraph of a graph $F$. The \textit{boundary} of $H$ in $F$ is $$\partial H=\left\{v\in H^0\ :\ \deg_F\left(v\right)>\deg_H\left(v\right)\right\}.$$
\end{defn}
\begin{lem}\label{lem:209}
Let $H\subset F$ be a subgraph of a finite leafless graph $F$ with no trivial components. Then: $$\chi\left(F\right)-\chi\left(H\right)\ \leq\ \dfrac{-1}{2}|\partial H|_0$$
\end{lem}
\begin{proof}
A graph $J$ satisfies $\chi\left(J\right)=\displaystyle\sum_{v\in J^0}\biggl(1-\dfrac{\deg\left(v\right)}{2}\bigg)$. We temporarily use $\chi$ to denote the number of vertices minus the number of open edges.\\ Let $J=\biggl(F-H\biggl)\displaystyle\bigcup_{v\in \partial H^0}S^1_v$ be obtained by removing $H$ and adding a circle at each vertex of $\partial H$. Then 
\begin{equation*}
\chi\left(F\right)-\chi\left(H\right)\ =\ \chi\left(F-H\right)\ =\ \chi\left(J\right)\ \leq\ -\dfrac{1}{2}|\partial H|_0.\ \qedhere
\end{equation*}
\end{proof}
\begin{lem}\label{lem:207}
Let $H$ be a subgraph of a finite graph $F$. Let $\psi:H\rightarrow F$ be a cellular immersion with $H\subset \psi\left(H\right)$. Suppose $H$ has no tree component and $\psi^{-1}\left(H\right)$ is \textcolor{black}{homeomorphic} to a forest. Then $\psi\left(T\right)\ \cap\ \partial H\neq \emptyset$ for each component $T\subset \psi^{-1}\left(H\right)$. Consequently, there exists $M\textcolor{black}{=M\left(F, H, \psi\right)}>0$ such that $|H|_1\ \leq\ M\ |\partial H|_0$.
\end{lem}
\begin{proof}
\textcolor{black}{Note that $\psi^{-1}\left(H\right)$ is not necessarily a subgraph of $F$. Each tree $T\subset \psi^{-1}\left(H\right)$ can be subdivided into a tree $\bar{T}$ so that $\psi|_{\bar{T}}$ is combinatorial}. Let $$d=\max\left\{\diam\left(\bar{T}\right)\ :\ T\subset \psi^{-1}\left(H\right),\ \textcolor{black}{\text{where}\ \bar{T}\ \text{is the subdivision of $T$}}\right\}$$ Since $H$ has no tree components, each component $T\subset \psi^{-1}\left(H\right)$ has a leaf that maps to $\partial H$.  So $H\subset \displaystyle\bigcup_{v\in \partial H}\mathcal{N}_d\left(v\right)$ where $\mathcal{N}_d\left(v\right)$ is a ball of radius $d$ centered at $v$. Let $M=\displaystyle\max\left\{|\mathcal{N}_d\left(v\right)|_1\ :\ v\in F^0\right\}$. Then $|H|_1\ \leq\ M |\partial H|_0$.
\end{proof}
\begin{defn}\label{defn:mappingtorus}
Let $F$ be a graph and let $H\subset F$ be a subgraph. The \textit{mapping torus} of a map $\psi:H\rightarrow F$ is the $2$-complex $X$ obtained as follows:
$$X=(F\sqcup (H\times [0,1]) )\ / \left\{\left(x,0\right)\sim x,\ \left(x,1\right)\sim\psi\left(x\right)\ :\ x\in H\right\}$$
The $2$-complex $X$ decomposes as a graph of spaces $X\rightarrow \Gamma_X$, where $\Gamma_X$ is a circle with one vertex $v$ and one edge $e$. Let $X_v=F$ and $X_e=H\times \left[0,1\right]$ be the vertex-space and edge-space, respectively, where $X_e$ is attached to $X_v$ via the maps $H\times \left\{0\right\}\rightarrow X_v$ and $H\times \left\{1\right\}\rightarrow X_v$. We refer to the images of $H\times \left\{0\right\}$ and $H\times \left\{1\right\}$ in $X_v$ as the \textit{outgoing} and \textit{incoming} edge-spaces, respectively. An edge $e$ of $X$ is \textit{vertical} if $e\subset F$, and \textit{horizontal} otherwise. Note that each vertex of $H$ gives rise to a horizontal edge of $X$, and each edge of $H$ gives rise to a $2$-cell of $X$. Moreover, each horizontal edge and each $2$-cell of $X$ arises in this manner.
\end{defn}
\begin{rem}\label{rem:common-claims}
Let $X$ be the mapping torus of a cellular immersion $\psi:H\rightarrow F$, where $H$ is a subgraph of a finite graph $F$. Let $Y\rightarrow X$ be a combinatorial immersion where $Y$ is a nontrivial compact, connected, and collapsed $2$-complex with no isolated edges. The decomposition $X\rightarrow \Gamma_X$ induces a decomposition $Y\rightarrow \Gamma_Y$ whose vertex-spaces are the components of the preimage of $F$ and whose \textcolor{black}{open} edge-spaces are the components of the preimage of $H\times (0,1)$. Let $Y_v$ be the disjoint union of the vertex-spaces, and let $Y_e\subset Y_v$ be the disjoint union of the outgoing edge-spaces. Then there is a \textcolor{black}{cellular immersion} $\Psi:Y_e\rightarrow Y_v$ whose mapping torus is $Y$ and the following diagram commutes:
\begin{center}
\begin{tikzcd}
Y_e \arrow[r, "\Psi"] \arrow[d] & Y_v \arrow[d] \\
H \arrow[r, "\psi"]  & F 
\end{tikzcd}
\end{center}
Define the \textit{boundary} of $Y$, denoted by $\partial Y$, as the union of the boundary vertices of $Y_e$ in $Y_v$. We make the following remarks:
\begin{enumerate}
\item\label{remark:07} The $2$-cells of $Y$ are in correspondence with the edges of $Y_e$.
\item\label{remark:01} Distinct outgoing edge-spaces in a vertex-space are disjoint. This holds since $Y\rightarrow X$ is an immersion and the outgoing edge-space in $X$ is an embedding. In particular, each edge of $Y_v$ is in at most one outgoing edge-space. 
\item\label{remark:02} Since $Y$ is collapsed and has no isolated edges, each edge in $Y_v$ lies in $\image\left(\Psi\right)$. Indeed, if there is a non-isolated edge $e\not\subset \image\left(\Psi\right)$, then by Remark~\eqref{remark:01}, $e$ lies in a unique outgoing edge-space. However, outgoing edge-spaces are embedded and so $e$ is a free face, contradicting that $Y$ is collapsed.
\item\label{remark:03} No edge-space of $Y$ has a leaf, since a leaf would give rise to a free face.
\item\label{remark:4.5} No edge-space (vertex-space) is a single vertex since otherwise $Y$ would have an isolated edge, a free face, or be trivial.
\item\label{remark:04} Outgoing edge-spaces are embeddings and $\Psi$ is an immersion since these mappings pull back from \textcolor{black}{the combinatorial immersion} $Y\rightarrow X$.
\item\label{remark:05.25} No vertex-space in $Y_v$ has a leaf. Indeed, by Remark~\eqref{remark:02}, each edge of $Y_v$ lies in an incoming edge-space. By Remark~\eqref{remark:03}, no edge-space has a leaf. By Remark~\eqref{remark:04}, the attaching maps of edge-spaces are immersions. Since the image of an immersed leafless graph contains no leafs, the claims holds. \textcolor{black}{Furthermore, by Remark~\eqref{remark:4.5}, no vertex-space of $Y$ is a tree, and so, $\chi\left(Y_{v_i}\right)\leq 0$ for all vertex-spaces $Y_{v_i}$ of $Y$.}

\end{enumerate}
\end{rem}
\begin{thm}\label{thm:malnormalhnn}
Let $H$ be a subgraph of a finite graph $F$. Let $X$ be the mapping torus of a cellular immersion $\psi:H\rightarrow F$. Suppose $\psi^{-1}\left(H\right)$ is \textcolor{black}{homeomorphic} to a forest. Then $X$ has negative immersions.
\end{thm}
\begin{proof}
Let $Y\rightarrow X$ be a combinatorial immersion where $Y$ is a nontrivial compact, connected, and collapsed $2$-complex with no isolated edges. As in Remark~\ref{rem:common-claims}, let $Y\rightarrow \Gamma_Y$ be the induced graph-of-spaces decomposition, and let $\Psi:Y_e\rightarrow Y_v$ be the map whose mapping torus is $Y$. By Remark~\ref{rem:common-claims}.\eqref{remark:02}, we have $Y_e\subset \image\left(\Psi\right)$. By Remark~\ref{rem:common-claims}.\eqref{remark:04}, the map $\Psi$ projects to $\psi$ and so $\Psi^{-1}\left(Y_e\right)$ is \textcolor{black}{homeomorphic to }a forest. \textcolor{black}{Each component $T'\subset \Psi^{-1}\left(Y_e\right)$ can be subdivided to form a tree $\bar{T}'$ so that $\Psi|_{\bar{T}'}$ is combinatorial}. Since $Y\rightarrow X$ is a combinatorial immersion, the \textcolor{black}{subdivided} trees of $\Psi^{-1}\left(Y_e\right)$ \textcolor{black}{embed} into the \textcolor{black}{subdivided trees of} $\psi^{-1}\left(H\right)$ \textcolor{black}{(as in Lemma~\ref{lem:207})}, and so for each component $T'\subset  \Psi^{-1}\left(Y_e\right)$, we have
$\diam\left(\bar{T}'\right)\ \leq \ d$, where $d=\max\left\{\diam\left(\bar{T}\right)\ :\ T\ \text{is a component in}\ \psi^{-1}\left(H\right)\right\}$. Moreover, since $X$ is compact, there is an upper bound $M=M\left(d\right)$ on the number of edges in any $d$-ball in $Y_v$. By Remarks~\ref{rem:common-claims}.\eqref{remark:03}-\eqref{remark:4.5}, $Y_e$ has no tree component. By Lemma~\ref{lem:207}, we have $|Y_e|_1\ \leq\ M|\partial Y_e|_0$. By Lemma~\ref{lem:209}, and Remark~\ref{rem:common-claims}.\eqref{remark:07}, we have:
\begin{equation*}
\chi\left(Y\right)\ =\ \chi\left(Y_v\right)-\chi\left(Y_e\right)\ \leq\ \dfrac{-1}{2}|\partial Y_e|_0\ \leq\ \dfrac{-1}{2M}|Y_e|_1\ =\  \dfrac{-1}{2M}|Y|_2 .\qedhere
\end{equation*}
\end{proof}
\section{Finite Height Mappings}\label{sec:finiteheight}

\begin{defn}\label{defn:composition} The \textit{generalized composition} of the functions $\alpha:A\rightarrow B$ and $\beta:C\rightarrow D$, where $C\subseteq B$, denoted by $\beta\bullet \alpha$, is $\beta\bullet \alpha=\beta\circ \alpha|_{\alpha^{-1}\left(C\right)}$.
\end{defn}

\begin{defn}\label{defn:combinatorial}
\textcolor{black}{Let $F$ be a connected graph and let $H\subset F$ be a subgraph.} Let $\psi:H\rightarrow F$ be a cellular immersion. For each $i\geq 0$, let $\psi^{i}$ denote the generalized composition of $\psi$ with itself $i$ times, where $\psi^0=id_F:F\rightarrow F$. Let $\psi^{-i}\left(H\right)=\left(\psi^{i}\right)^{-1}\left(H\right)$.

Let $Z_i$ denote the domain of $\psi^{i}$. Then $Z_{i+1}=\left\{x\in Z_{i}\ :\ \psi^{i}\left(x\right)\in H\right\}\textcolor{black}{=\psi^{-i}\left(H\right)}$, \textcolor{black}{for each $i\geq 0$}. The \textit{combinatorial domain} $D_i$ of $\psi^{i}$ is the largest subgraph in $Z_i$. Note that $Z_i$ is not necessarily a subgraph of $F$, $Z_{i+1}\subseteq Z_i$, \textcolor{black}{and $D_{i+1}\subseteq D_i$} for all $i\geq 0$. \textcolor{black}{Moreover, $Z_i$ has a part that deformation retracts to $D_i$ and a part that is a disjoint union of closed intervals and singletons. Thus, when $Z_i$ is not homeomorphic to a forest, at least one component of $D_i$ is not a tree}. Let $D_\infty\subset H$ be the subgraph whose edges and vertices map into $H$ under all powers of $\psi$. Note that $\emptyset\subseteq D_\infty\subseteq D_{i+1}\subseteq D_i$. 

The \textit{directed height} of $\psi$ is:
$$\overrightarrow{\height}\left(\psi\right)=\inf\left\{i :\ \psi^{-i}\left(H\right)\ \text{is a forest}\right\}$$ 
Note that $\overrightarrow{\height}\left(\psi\right)=0$ if and only if $H$ is a forest. We use the following notation: $$\lVert \psi \rVert=\max\left\{|\psi\left(e\right)|_1\ :\ e\subset H^{1}\right\}$$
\end{defn}
\begin{rem}\label{rem:bound on paths in bass serre tree}
\textcolor{black}{ $\overrightarrow{\height}\left(\psi\right)=\ell<\infty$ if and only if} the length of embedded directed paths in the Bass-Serre tree with infinite stabilizers is bounded by $\ell$. Note that the Bass-Serre tree is directed because of the map to the underlying graph of the HNN extension which is a directed loop.
\end{rem}
\begin{defn}
 The \textit{height} of a subgroup $\mathcal{H}$ in $\mathcal{G}$, denoted by $\height\left(\mathcal{H}\right)$, is the supremal number of distinct cosets $\left\{\mathcal{H}g_i\right\}_{i\in I}$ such that $\displaystyle\bigcap_{i\in I} \mathcal{H}^{g_i}$ is infinite.
\end{defn}
\begin{lem}\label{lem:heights are the same}
Let $H$ be a subgraph of a finite graph $F$ and let $\psi:H\rightarrow F$ be a cellular immersion. Let $X$ be the mapping torus of $\psi$. Then $\pi_1H$ has finite height in $\pi_1X$ if and only if $\psi$ has finite directed height.
\end{lem}
\begin{proof}
Let $\mathcal{H}=\pi_1H$ and $\mathcal{F}=\pi_1F$. Suppose $\mathcal{H}$ has finite height in $\pi_1X$. Then $\height\left(\mathcal{H}\right)$ bounds the number of distinct cosets $\left\{\mathcal{H}g_i\right\}$  such that $|\mathcal{H}^{g_1}\cap\cdots\cap\mathcal{H}^{g_n}|$ is infinite. So the number of edges in the Bass-Serre tree $T$ with a common infinite stabilizer is likewise bounded.  Hence $\height\left(\mathcal{H}\right)$ bounds the length of embedded paths in $T$ with infinite stabilizer. Thus $\overrightarrow{\height}\left(\psi\right)<\infty$.

Suppose $\overrightarrow{\height}\left(\psi\right)<\infty$. Then $\overrightarrow{\height}\left(\psi\right)$ bounds the length of embedded paths in $T$ with infinite stabilizers. There is a uniform upper bound on the degree of vertices of any subtree $T'\subset T$ with point-wise stabilizer of $T'$ infinite. Indeed, the number of incoming edges of each vertex in $T'$ is bounded by $r=\height\left(\psi_*\left(H\right)\right)$ in $\mathcal{F}$, since every finitely generated subgroup of a free group has finite height  \cite{GMRS98}. 
Thus $T'$ is a rooted tree of length $\leq \overrightarrow{\height}\left(\psi\right)$ and incoming degree $\leq r$. So the number of edges in $T'$ is $\leq r^{\overrightarrow{\height}\left(\psi\right)}$. Any set of cosets of the edge group corresponds to a set of edges in $T$. The intersection of the corresponding conjugates point-wise stabilizes those edges, and thus point-wise stabilizes the smallest tree $T'$ containing them. Hence the number of cosets is bounded by $\leq r^{\overrightarrow{\height}\left(\psi\right)}$. 
\end{proof}
\begin{lem}\label{lem}\label{lem:nested domains}
Let $H$ be a subgraph of $F$. Let $\psi:H\rightarrow F$ be a cellular immersion with $\overrightarrow{\height}\left(\psi\right)=\ell<\infty$. Then $D_\infty$ is a (possibly empty) forest.
\end{lem}
\begin{proof} 
We have $D_\infty\subseteq D_{\ell+1}$. So, it suffices to show that $D_{\ell+1}$ is a forest. Suppose $C\subset D_{\ell+1}$ is an embedded circle. Then $\psi^\ell\left(C\right)\subset H$ and so $\psi^{-\ell}\left(H\right)$ is not a forest, contradicting the assumption.
\end{proof}
\begin{lem}\label{lem:infinite height give zer euler}
Let $F$ be a \textcolor{black}{connected} graph and let $ H\subset F$ be a finite subgraph. Let $X$ be the mapping torus of a cellular immersion $\ \psi:H\rightarrow F$. If $\psi$ has infinite directed height, then $H$ contains a connected subgraph $D\subset H$ with $\chi\left(D\right)\ \textcolor{black}{\leq}\ 0$ such that $\psi\left(D\right)= D$. Consequently, $X$ contains a subcomplex $Y\hookrightarrow X$, where $Y$ is a connected, compact, and collapsed $2$-complex with no isolated edges and $\chi\left(Y\right)=0$.
\end{lem}
\begin{proof}
Since $\psi$ has infinite directed height, for each $i\geq 0$, we have $\psi^{-i}\left(H\right)$ is not a forest. So each ${D}_i$ contains an embedded circle. Since $H$ is finite and $D_{i+1}\subseteq D_i$, there is an integer $p$ such that for all $j>p$ we have $D_ {j+1}=D_j$. \textcolor{black}{Then $D_j$ contains a component $D$ with $\chi\left(D\right)\leq 0$ and  $\psi\left(D\right)= D$. In particular, since $\psi$ is an immersion, $\psi\left(D_{\text{core}}\right)=D_{\text{core}}$, where $D_{\text{core}}$ is the core of $D$}. The mapping torus of $\psi$ restricted to $D_{\text{core}}$ provides $Y$.
\end{proof}
\begin{defn}\label{defn:ladders and fans} Let $Q\rightarrow \Gamma_Q$ be a graph of spaces where $\Gamma_Q$ is equal to a subdivided interval $[0,k]$ directed from $0$ to $\textcolor{black}{1\ \leq}\ k\ \textcolor{black}{\leq \infty}$. Suppose each vertex-space $Q_{v_i}$ is a tree where $Q_{v_0}$ has exactly one edge $f_0$. For each edge-space $Q_{e_i}\times I$ there is an \textit{outgoing} attaching map $Q_{e_i}\times \left\{0\right\}\rightarrow Q_{v_{i-1}}$ and an \textit{incoming} attaching map $Q_{e_i}\times \left\{1\right\}\rightarrow Q_{v_i}$. When each outgoing attaching map is an embedding onto a single edge $f$ of the vertex-space, then $Q$ is a \textit{ladder} and $f$ is a \textit{connecting edge}. When each attaching map is \textcolor{black}{bijective}, then $Q$ is a \textit{fan}. The \textit{rim} of a fan $Q$, denoted by $\rim\left(Q\right)$, is $Q_{v_k}$. The \textit{length} of $Q$ is $\length\left(Q\right)=k$. \textcolor{black}{We allow the case $k=\infty$ and say that $Q$ is an \textit{infinite ladder/fan}}. We say $Q$ \textit{arises} from $f_0$, and $f_0$ \textit{gives rise} to $Q$.

The space $Q$ is a cell complex as follows: we have already declared each $Q_{v_i}$ is a tree and so it remains to describe the additional $1$-cells and $2$-cells of $Q$. Each open edge-space $Q_{e_i}\times (0,1)$ has a product structure induced by the graph $Q_{e_i}$. See Figure~\ref{image2.png}.
\begin{figure}\centering
\includegraphics[width=.6\textwidth]{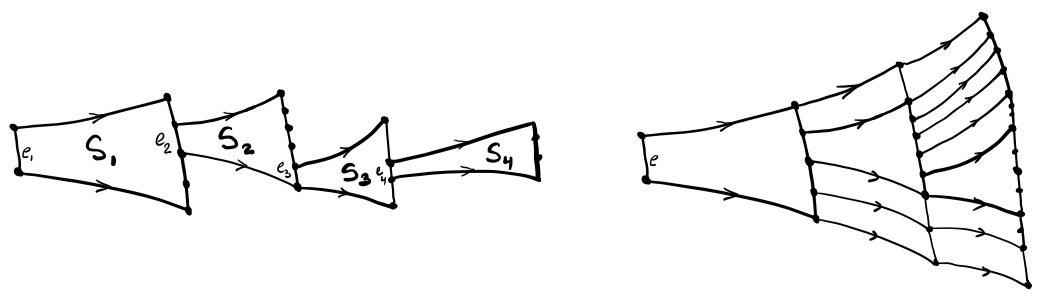}
\caption[]{\label{image2.png}
Left: A ladder of length $4$ emerging from $e_1$. Right: A fan of length $3$ emerging from $e$.}
\end{figure} 
The edges in the vertex-spaces are \textit{vertical} and the remaining ones are \textit{horizontal}. Each vertex in the image of $Q_{e_i}\rightarrow Q_{v_{i-1}}$ \textit{gives rise} to a horizontal edge in $Q$. Each edge $f$ in the image of $Q_{e_i}\rightarrow Q_{v_{i-1}}$ \textit{gives rise} to a $2$-cell $S\subset Q$. We say $S$ \textit{arises} from $f$.  

Let $X$ be a \textcolor{black}{$2$-}complex with a graph-of-spaces structure whose $1$-skeleton is partitioned into horizontal and vertical edges, \textcolor{black}{where the vertical edges are the edges of vertex-spaces, and the horizontal edges are the remaining ones}. An \textit{immersed ladder} of $X$ is a combinatorial immersion $\lambda:L\rightarrow X$ that maps vertical/horizontal edges of a ladder $L$ to vertical/horizontal edges of $X$. An \textit{immersed fan} $\phi: Q\rightarrow X$ is defined analogously. An edge $e\subset X$ has a $k$-ladder (resp. $k$-fan), if there is an immersed ladder $\lambda:L\rightarrow X$ (resp. immersed fan $\phi:Q\rightarrow X$) of length $k$ emerging from $e'$ such that $\lambda\left(e'\right)=e$ (resp. $\phi\left(e'\right)=e$). When $X$ is the mapping torus of $\psi:H\rightarrow F$, we require that immersions preserve the orientation of horizontal edges. 

\textcolor{black}{Let $X$ be the mapping torus of $\psi: H\rightarrow F$}. Let $H_i=\overline{D_i-D_{i+1}}$ be the subgraph whose edges give rise to $i$-fans but not $(i+1)$-fans. When $H_i=\emptyset$, we have $D_i=D_{i+1}=D_\infty$ is the subgraph whose edges give rise to infinite fans. Then $D_\infty$ is $\psi$-invariant. Let $m=m(\psi)$ denote the supremum of lengths of maximal finite fans in $X$. Note that when $H$ is finite we have $m<\infty$ \textcolor{black}{since any maximal finite fan is determined by the edge it arises from.}
\end{defn}
\begin{lem}\label{lem:fans}
Let $H$ be a subgraph of a finite \textcolor{black}{connected} graph $F$. Let $X$ be the mapping torus of a cellular immersion $\psi: H\rightarrow F$ with $\overrightarrow{\height}\left(\psi\right)<\infty$. Let $m=m\left(\psi\right)$ be the maximal length of \textcolor{black}{immersed} finite fans in $X$. Let $Y\rightarrow X$ be a combinatorial immersion, where $Y$ is a nontrivial, compact, connected, and collapsed $2$-complex with no isolated edges. Let $Y\rightarrow \Gamma_Y$ be the induced graph-of-spaces decomposition and let $\partial Y$ be the associated boundary. Then there exists $M\textcolor{black}{=M\left(H, F, \psi\right)}>0$ such that each $2$-cell $S$ of $Y$ lies in \textcolor{black}{the image of an immersed} ladder $\lambda:L\rightarrow Y$ with $\length\left(L\right)\leq m+1$ emerging from $e$ where $\dist\left(\lambda\left(e\right),\partial Y\right)\leq M$.
\end{lem}
\begin{proof} Let $S$ be a $2$-cell of $Y$. Since $Y$ is collapsed, there is an immersed ladder $\lambda:L\rightarrow Y$ with $\length\left(L\right)= m+1$ and whose $(m+1)$-th $2$-cell maps to $S$. Let $\left\{e_1,\ldots,e_{m+1}\right\}$ be the connecting edges of $L$. For $1\leq i\leq m+1$, let $O_{v_i}$ be the outgoing edge-space containing $\lambda\left(e_i\right)$ and let $Y_{v_i}$ be the vertex-space containing $O_{v_i}$. Let $\phi:Q\rightarrow Y$ be the maximal immersed fan emerging from $\lambda\left(e_1\right)$. See Figure~\ref{image7.png}. 
\begin{figure}\centering
\includegraphics[width=.9\textwidth]{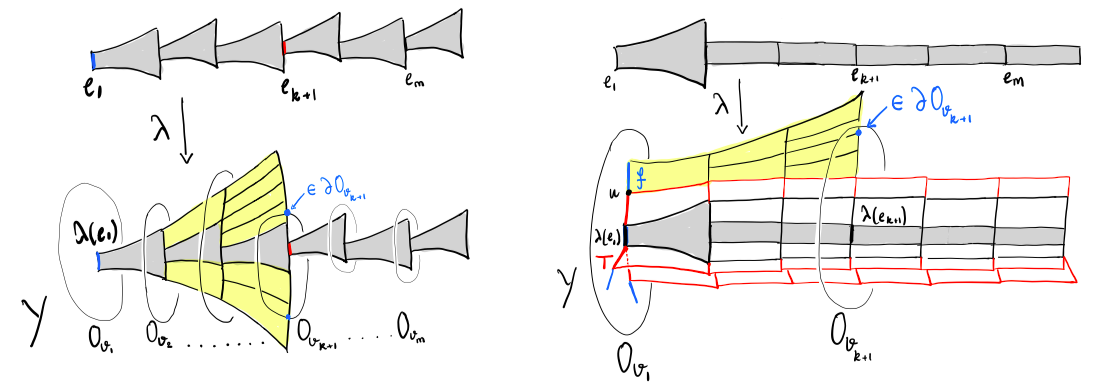}
\caption[]{\label{image7.png}
On the left: Case $1$, and on the right case $2$.}
\end{figure}

\textbf{Case 1}: If $\length\left(Q\right)=k\leq m$, then $\lambda\left(e_{k+1}\right)\ \subset \ \phi\left(\rim\left(Q\right)\right)\cap O_{v_{k+1}}$. Since $Q$ is maximal, $\image\left(\rim\left(Q\right)\rightarrow Y_{v_{k+1}}\right)\ \not\subset\ O_{v_{k+1}}$, and so $\phi\left(\rim\left(Q\right)\right)\ \cap\ \partial O_{v_{k+1}}\neq \emptyset$. Since fans in $Y$ project to fans in $X$, we have $|\rim\left(Q\right)|_1\leq \lVert \psi \rVert^m$. Thus, $\dist\left(\lambda\left(e_{k+1}\right),\partial Y\right)\leq \lVert \psi \rVert^m$.

\textbf{Case 2}: If $\length\left(Q\right)> m$, then the image of $Q\rightarrow Y\rightarrow X$ \textcolor{black}{is} an infinite fan of $X$. Let $T\subset O_{v_1}$ be the maximal connected subgraph containing $\lambda\left(e_1\right)$ and whose edges give rise to $(m+1)$-fans in $Y$. Hence $T$ immerses in $D_\infty$. \textcolor{black}{Since $\overrightarrow{\height}\left(\psi\right)<\infty$, it follows from  Lemma~\ref{lem:nested domains} that $D_\infty$ is a forest. So} $T$ is a tree with $\diam\left(T\right)\leq \diam\left(D_\infty\right)$. Let $u\in T$ be a leaf. Since $Y$ is collapsed, outgoing edge-spaces have no leaves. So there is an edge $f\subset O_{v_1}$ containing $u$ with $f\not\subset T$. By maximality of $T$, the maximal fan $\phi'\left(Q'\right)$ emerging from $f$ has length $k\leq m$. So $\phi'\left(\rim\left(Q'\right)\right)\ \cap\ \partial O_{v_{k+1}}\neq \emptyset$. Hence, $\dist\left(\lambda\left(e_{k+1}\right),\partial Y\right)\leq \diam\left(D_\infty\right)+\lVert \psi \rVert^m$.

The claim follows with $M=\diam\left(D_\infty\right)+\lVert \psi \rVert^m$.
\end{proof}
\begin{thm}\label{thm:main coherence theorem}
Let $F$ be a finite \textcolor{black}{connected} graph and let $H\subset F$ be a subgraph. Let $X$ be the mapping torus of a cellular immersion $\psi:H\rightarrow F$. Then $X$ has negative immersions if and only if $\psi$ has finite directed height.
\end{thm}
\begin{proof}
The ``only if'' direction holds by Lemma~\ref{lem:infinite height give zer euler}.

Suppose $\overrightarrow{\height}\left(\psi\right)<\infty$. Let $Y\rightarrow X$ be a combinatorial immersion where $Y$ is a nontrivial compact, connected, and collapsed $2$-complex with no isolated edges. Let $Y\rightarrow \Gamma_Y$ be the induced graph-of-spaces decomposition. For each $v\in \Gamma_Y^{0}$, let $Y_v$ be the corresponding vertex-space and let $O_v$ be the disjoint union of outgoing edge-spaces in $Y_v$. Let $m=m\left(\psi\right)$ be the supremal length of maximal finite fans in $X$. By Lemma~\ref{lem:fans}, there exists $M>0$ such that each $2$-cell of $Y$ lies in a ladder of length $\leq m+1$ emerging from a vertical edge $e$ 
with $\dist\left(e,\partial Y\right)\leq M$. Let $\partial' Y$ be the set of boundary points of $Y$ that are at a distance $\leq M$ from \textcolor{black}{such} edges $e$. So $\partial' Y\subseteq \partial Y=\displaystyle\bigsqcup_{v\in \Gamma_Y^{0}}\partial O_v$. Since $Y\rightarrow X$ is a combinatorial immersion, there is an upper bound $N$ on the number of edges in an $M$-ball in the vertex-spaces of $Y$. Note that $N=N\left(F,M\right)$ is a function of $F$ and $M$. Consider the $M$-balls centered at vertices of $\partial' Y$. In each such ball, there are at most $N$ edges and each edge gives rise to at most $\lVert \psi \rVert^{m}$ ladders of length $\leq\ (m+1)$. The number of $2$-cells in each ladder is $\leq\ (m+1)$. Then:
$$|Y|_2\ \leq\ \displaystyle\sum_{v\in \partial' Y}(m+1)\lVert \psi \rVert^m N\ =\ (m+1)\lVert \psi \rVert^m N|\partial' Y|_0\ \leq\ (m+1)\lVert \psi \rVert^m N|\partial Y|_0$$
and so $$\dfrac{|Y|_2}{(m+1)\lVert \psi \rVert^m N}\ \leq\ |\partial Y|_0$$


By Remark~\ref{rem:common-claims}.\eqref{remark:05.25}, the vertex-spaces of $Y$ have no leaves. Then the conclusion holds by the following double inequality. Its first equality is straightforward. Its last inequality follows from above, and its middle inequality holds by Lemma~\ref{lem:209}. 
\begin{equation*}
\chi\left(Y\right)\ =\ \displaystyle\sum_{v\in \Gamma_Y^{0}}\left(\chi\left(Y_v\right)-\chi\left(O_v\right)\right)
\ \leq\ \dfrac{-1}{2}|\partial Y|_0\ \leq\ \dfrac{-1}{2(m+1)\lVert \psi \rVert^m N}|Y|_2 \qedhere
\end{equation*}
\end{proof}
\begin{defn}\label{defn:connection to irreduciible}
Let $\mathcal{F}$ be a free group. 
There is a natural generalization of fully irreducible endomorphisms of free groups to fully irreducible partial endomorphisms. 
 A partial endomorphism $\psi: \mathcal{H}\rightarrow \mathcal{F}$ is \textit{fully irreducible} if there does not exist $n>0$, a proper free factor $\mathcal{H}'\subset \mathcal{H}$, and $g\in \mathcal{F}$ such that $\psi^n\left(\mathcal{H}'\right)\subset g^{-1}\mathcal{H}'g$. See Definition~\ref{defn:composition} for the notion of \textit{generalized composition} explaining $\psi^n$. 
The standard notion of fully irreducible endomorphism focuses on the case where $\mathcal{H}=\mathcal{F}$ \cite{MR1147956}. 

\textcolor{black}{The \textit{standard $2$-complex} associated to a presentation of a group $G$ is a 2-dimensional cell complex formed by a single vertex, one circle at the vertex for each generator of $G$, and a $2$-cell for each relation in the presentation. The attaching maps of the $2$-cells are determined by the presentation.}
\end{defn}
In the language of Definition~\ref{defn:connection to irreduciible}, our result  shows  the following:
\begin{thm}\label{thm:irreducible is finite height}
Let $\mathcal{H}$ be a proper free factor of a finitely generated free group $\mathcal{F}$, and let $\psi:\mathcal{H}\rightarrow \mathcal{F}$ be a monomorphism. Let $X$ be the standard $2$-complex of the HNN extension of $\mathcal{F}$ with respect to $\psi$. Then $X$ has negative immersions if and only if $\psi$ is fully irreducible.  
\end{thm} 
\begin{proof}
The proof follows from Lemma~\ref{lem:heights are the same}, Lemma~\ref{lem:nested domains}, Lemma~\ref{lem:infinite height give zer euler}, and Lemma~\ref{lem:fans}. \end{proof}
\begin{rem}\label{rem:c is less than one}
If $c=\dfrac{1}{2(m+1)\lVert \psi \rVert^m N}$  is the constant in $\chi\left(Y\right)\ \leq \ -c|Y|_2$, then $0<c<1$.
\end{rem}
\begin{rem}\label{rem:it works with isolated edges} 
\textcolor{black}{In the proof of Theorem~\ref{thm:main coherence theorem}, we assume that $Y$ has no isolated edges, as required by the definition of Negative Immersions. However, the claim that $\chi\left(Y\right)\leq -c|Y|_2$ holds even if we  allow $Y$ to have isolated edges.} This follows from a simple induction on the number of isolated edges in $Y$. Indeed, the base case holds by Theorem~\ref{thm:main coherence theorem}. Now, let $e$ be an isolated edge of $Y$. Then either $e$ is not separating and $Y=Y_1\cup e$, or $e$ is separating and $Y=Y_1\cup e \cup Y_2$. In the former case, we have
$$\chi\left(Y\right)< \chi\left(Y-e\right)=\chi\left(Y_1\right)\leq -c|Y_1|_2=-c|Y|_2$$ where the last inequality holds by induction. In the latter case, we have $$\chi\left(Y\right)=\chi\left(Y_1\right)+\chi\left(Y_2\right)-1<\chi\left(Y_1\right)+\chi\left(Y_2\right)\leq -c_1|Y_1|_2+c_2|Y_2|_2\ \leq\ -c|Y|_2$$
where the last inequality holds by induction, and $c=\min\left\{c_1, c_2\right\}$.
\end{rem}
Motivated by our desire to verify Property~\ref{pr:fgip relative to H} of the next section, we note the following consequence of the preceding statements. This does not prove Property~\ref{pr:fgip relative to H} since it does not assert that the edge groups in the splitting of $\mathcal{K}$ equal the intersections of $\mathcal{K}\cap \mathcal{H}^g$, for $g\in \pi_1X$. 
\begin{thm}\label{thm:finite intersection property}
Let $F$ be a finite \textcolor{black}{connected} graph and let $H\subset F$ be a subgraph. Let $X$ be the mapping torus of a cellular immersion $\psi:H\rightarrow F$ with $\overrightarrow{\height}\left(\psi\right)<\infty$. Let $\mathcal{K}\subset \pi_1 X$ be a finitely generated subgroup. Then $\mathcal{K}$ splits over edge groups with uniformly bounded Euler characteristic.
\end{thm}
\begin{proof}
By Theorem~\ref{thm:main coherence theorem}, $X$ has negative immersions. By Theorem~\ref{thm:negative Imm give coherence}, $\pi_1 X$ is coherent. So there is a combinatorial immersion $Y\rightarrow X$, with $\pi_1 Y\xrightarrow{\simeq}\mathcal{K}$ where $Y$ is compact and connected. We can assume that $Y$ is collapsed since collapsing is a homotopy equivalence. Let $Y\rightarrow \Gamma_Y$ be the graph-of-spaces decomposition induced by the decomposition $X\rightarrow \Gamma_X$. Let $V_Y$ and $O_Y$ be the disjoint union of vertex-spaces and outgoing edge-spaces of $Y$, respectively. We show that $\chi\left(O_Y\right)$ is uniformly bounded by a function of $\rank\left(\pi_1 Y\right)$. In fact, we show that $\chi\left(O_Y\right)\ \geq\ \dfrac{\chi\left(Y\right)}{c}$. By Theorem~\ref{thm:main coherence theorem}, Remark~\ref{rem:c is less than one} and Remark~\ref{rem:it works with isolated edges}, there is a constant $c\in (0,1)$ such that $\chi\left(V_Y\right)-\chi\left(O_Y\right)\ \leq\ -c|Y|_2$. So $\chi\left(O_Y\right)\ \geq\ \chi\left(V_Y\right)+c|Y|_2$. We have $\chi\left(Y\right)\ =\ \chi\left(V_Y\right)-E+|Y|_2$, where $E$ are the number of the horizontal edges in $Y$. Since $c-1<0$, we have $$\chi\left(O_Y\right)\ \geq\ \chi\left(Y\right)+E-|Y|_2+c|Y|_2\ \geq\ \chi\left(Y\right)+(c-1)|Y|_2\ \geq\ \dfrac{\chi\left(Y\right)}{c}$$
where the last inequality follows by By Theorem~\ref{thm:main coherence theorem}. 
\end{proof}
\section{generalization to $\pi_1$-injective mappings}\label{sec:generalization to pi1 mappings}

Theorem~\ref{thm:main coherence theorem} generalizes to $\pi_1$-injective mappings. For this, we require the image of $\psi$ to be a subgraph. Moreover, by slightly changing the definition of directed height, we are able to to use the proof for the immersion case in the $\pi_1$-injective case. To this end, we use Theorem~\ref{thm:stallings} to factorize $\psi^{i}$ into an immersion post-composed with a $\pi_1$-isomorphism. Then generalized versions of Lemma~\ref{lem:nested domains} and Lemma~\ref{lem:infinite height give zer euler} can be proved with respect to the immersive factor of $\psi^{i}$, whereas Lemma~\ref{lem:heights are the same} and Lemma~\ref{lem:fans} remain true in both cases. We note that fans behave differently in this case. Indeed, in the immersion case, infinite directed height implies the existence of infinite fans. This is not true in the non-immersion case. See Figure~\ref{image8.png}. However, when the directed height is finite, fans have the expected behaviour, namely, infinite fans arise from tree components of $H$ and maximal finite fans have uniformly bounded length $m$.
\begin{figure}\centering
\includegraphics[width=.2\textwidth]{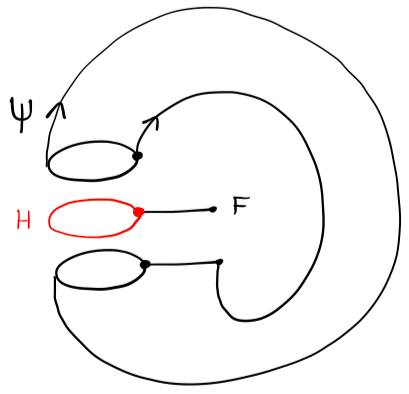}
\caption[]{\label{image8.png}
$\psi$ has infinite directed height, but the mapping torus of $\psi$ has one fan of maximal length $1$.}
\end{figure}
\begin{thm}[Stallings Factorization~\cite{Stallings83}]\label{thm:stallings}
Let $H$ and $G$ be graphs. Every cellular map $\psi:H\rightarrow G$ factors as $\psi=\theta\ \circ\  \rho$ where $\rho$ is a composition of edge collapses and foldings, and $\theta$ is an immersion. Moreover, $\rho$ is a homotopy equivalence if and only if $\psi$ is $\pi_1$-injective.
\end{thm}

Given a $\pi_1$-injective map $\psi:H\rightarrow F$, where $\psi\left(H\right)$ is a subgraph of $F$, define $\psi^{i}$ as in Section~\ref{sec:finiteheight}. Note that $\psi^{i}$ is $\pi_1$-injective for all $i\geq 0$. Then by Theorem~\ref{thm:stallings}, we have $\psi^{i}=\theta_i\ \circ\ \rho_i$ where $\rho_i$ is a $\pi_1$-isomorphism and $\theta_i$ is an immersion:
\begin{center}
  \begin{tikzcd}
Z_i \arrow[rr, "\psi^{i}"] \arrow[rd, "\rho_i", two heads] &                                   & F \\
                                                   & \rho\left(Z_i\right) \arrow[ru,loop-math to, "\theta_i"] &  
\end{tikzcd}
\end{center}
Define the \textit{directed height} of $\psi$ as:
$\overrightarrow{\height}\left(\psi\right)=\inf\left\{i :\  \theta_i^{-1}\left(H\right)\ \text{is a forest}\right\}$. 
\begin{lem}\label{lem}\label{lem:nested domains nonimmersion}
 Let $H$ be a subgraph of $F$. Let $\psi:H\rightarrow F$ be a $\pi_1$-injective cellular   map with $\overrightarrow{\height}\left(\psi\right)=\ell<\infty$. Then $D_\infty$ is a (possibly empty) forest.
\end{lem}
\begin{proof} 
We have $D_\infty\subseteq  D_{\ell+1}$. So, it suffices to show that $D_{\ell+1}$ is a forest. Suppose $C\subset D_{\ell+1}$ is an embedded circle. Then $\psi^\ell\left(C\right)\subset H$. Since $\rho_\ell$ is $\pi_1$-injective, $\rho_\ell\left(C\right)$ is not a tree. But  $\rho_\ell\left(C\right)\subset \theta_\ell^{-1}\left(H\right)$, and so   $\theta_\ell^{-1}\left(H\right)$ is not a forest, which is a contradiction.
\end{proof}

\begin{lem}\label{lem:infinite height give zer euler nonimmersion}
Let $F$ be a graph and let $ H\subset F$ be a finite subgraph. Let $X$ be the mapping torus of a $\pi_1$-injective cellular map $\ \psi:H\rightarrow F$ where $\image\left(\psi\right)$ is a subgraph of $F$. If  
$\overrightarrow{\height}\left(\psi\right)=\infty$, then there is a subcomplex $Y\subset X$ with $\chi\left(Y\right)=0$ and $Y$ is connected, compact, collapsed, and has no isolated edges.
\end{lem}
\begin{proof}
Let $A_i=\image\left(\psi^{i}\right)\cap H$ for $i\geq 0$. Then by definition, $A_{i+1}\subseteq A_i$ and $\psi\left(A_i\right)=A_{i+1}$. Moreover, $\rank\left(A_i\right)>0$. Indeed, since $\theta_i$ is an immersion, $\theta_i^{-1}\left(H\right)$ is a forest whenever $A_i$ is a forest. By assumption, $\overrightarrow{\height}\left(\psi\right)=\infty$, and so   $\theta_i^{-1}\left(H\right)$ is not a forest for all $i\geq 0$. Since $H$ is a finite graph, there exist $j<k$ such that $A_j=A_k$. Then $A_j=A_k\subseteq A_{k-1}\subseteq \cdots\subseteq A_{j+1}\subseteq A_j$ and so $A_j=A_{j+1}$. Since $\psi$ is $\pi_1$-injective and $\rank\left(A_j\right)>0$, a component of the mapping torus of $\psi$ restricted to the core of the non-tree components of $A_j$ provides $Y$.
\end{proof}

\begin{rem}\label{rem:pi1-injectivity needed} The requirement that $\psi$ be $\pi_1$-injective in Lemma~\ref{lem:infinite height give zer euler nonimmersion} ensures that finite directed height implies negative immersions. Indeed, let $X$ be the mapping torus of $\psi:H\twoheadrightarrow H$ where $H$ is a connected leafless graph of positive rank. Suppose $\psi_*\left(\pi_1H\right)$ is trivial. Then $\overrightarrow{\height}\left(\psi\right)=1$ but $X$ is collapsed with $\chi\left(X\right)=0$, and so the negative immersions property fails. See Figure~\ref{image9.png}.
\begin{figure}\centering
\includegraphics[width=.4\textwidth]{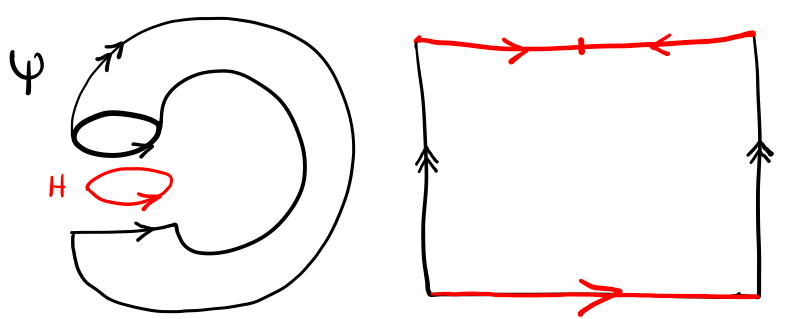}
\caption[]{\label{image9.png}
The mapping torus of $\psi$ has zero Euler characteristic and $\overrightarrow{\height}\left(\psi\right)=1$. }
\end{figure}
\end{rem}

\begin{thm}\label{thm:main coherence theorem nonimmersion}
Let $F$ be a finite graph and let $H\subset F$ be a subgraph. Let $X$ be the mapping torus of a $\pi_1$-injective cellular map $\psi:H\rightarrow F$. Then $X$ has negative immersions if and only if $\psi$ has finite directed height.
\end{thm}

\section{Discussion of Related Properties}\label{sec:related properties}
Let $H$ be a subgraph of a finite graph $F$ and let $\mathcal{H}=\pi_1H$ and $\mathcal{F}=\pi_1F$. Let $X$ be the mapping torus of a cellular immersion $\psi:H\rightarrow F$. Consider the following properties: 

\begin{enumerate}
\item\label{pr:localQC} $\pi_1X$ is locally quasiconvex.
\item\label{pr:base is QC} $\mathcal{F}$ and $\mathcal{H}$ are quasiconvex.
\item\label{pr:base is finite height} $\mathcal{F}$ and $\mathcal{H}$ have finite height.
\item\label{pr:fgip} $\pi_1X$ has the finitely generated intersection property.
\item\label{pr:negative immersions} 
$X$ has negative immersions.
\item\label{pr:no ascending hnn ext} $\pi_1X$ contains no subgroup isomorphic to an ascending HNN extension of a finitely generated free group.
\item\label{pr:hyperbolicity} $\pi_1X$ is hyperbolic.
\item\label{pr:no BS} $\pi_1X$ contains no $\BS\left(1,m\right)$ for $m>0$.
\item\label{pr:The group has a quasiconvex hierarchy} $\pi_1X$ has a quasiconvex hierarchy.
\item\label{pr:virtual special} $\pi_1X$ is virtually special.

\item\label{pr:fgip relative to H} $\mathcal{K}\cap \mathcal{H}$ is finitely generated whenever $\mathcal{K}\subset \pi_1X$ is finitely generated.
\item\label{pr:tame generation} Each finitely generated subgroup of $\pi_1X$ is tamely generated.
 \end{enumerate}

\eqref{pr:localQC}$\Rightarrow$ \eqref{pr:base is QC} is immediate. When $\pi_1X$ is hyperbolic, we have \eqref{pr:base is QC} $\iff$ \eqref{pr:base is finite height}, where $(\Rightarrow)$ holds by \cite{GMRS98} and $(\Leftarrow)$ holds by \cite{Mitra2004}. A group has the \textit{ﬁnitely generated intersection property} (FGIP) if the intersection of any two ﬁnitely generated subgroups is also ﬁnitely generated. For instance, free groups have the FGIP \cite{Howson54}. \eqref{pr:fgip}$\Rightarrow$\eqref{pr:no ascending hnn ext} holds by \cite{MR4443138} and \eqref{pr:localQC} $\Rightarrow$ \eqref{pr:fgip} holds by \cite{Short91}. \eqref{pr:negative immersions}$\Rightarrow$\eqref{pr:no ascending hnn ext} since ascending HNN extensions of free groups have Euler characteristic zero, and \eqref{pr:no ascending hnn ext}$\Rightarrow$\eqref{pr:base is finite height} by Lemma~\ref{lem:heights are the same} and Lemma~\ref{lem:infinite height give zer euler}. \eqref{pr:negative immersions}$\iff$\eqref{pr:base is finite height} holds by Theorem~\ref{thm:main coherence theorem}, and \eqref{pr:negative immersions}$\Rightarrow$\eqref{pr:no BS} is a special case of \eqref{pr:negative immersions}$\Rightarrow$\eqref{pr:no ascending hnn ext}. It is well known that \eqref{pr:hyperbolicity}$\Rightarrow$\eqref{pr:no BS}, e.g \cite{MR1170363}. \eqref{pr:hyperbolicity} $+$ \eqref{pr:The group has a quasiconvex hierarchy}$\Rightarrow$\eqref{pr:virtual special} by \cite{WiseCBMS2012}. \eqref{pr:fgip relative to H}$\Rightarrow$\eqref{pr:tame generation} since if $\mathcal{K}\cap \mathcal{H}$ is finitely generated for each finitely generated $\mathcal{K}$, then $\mathcal{K}^g\cap \mathcal{H}$ is finitely generated for each $g$, and so $\mathcal{K}\cap \mathcal{H}^g$ is finitely generated for each $g$. See \cite{MR3079269} for the definition of ``tamely generated''. \eqref{pr:tame generation}$\Rightarrow$ \eqref{pr:localQC} holds by \cite{MR3079269}. \eqref{pr:negative immersions} $\Rightarrow$ \eqref{pr:The group has a quasiconvex hierarchy} holds by the following argument: \eqref{pr:negative immersions} $\Rightarrow$ \eqref{pr:no ascending hnn ext} and by \cite{CW22}, we have \eqref{pr:no ascending hnn ext}$\Rightarrow$ $\pi_1X\subset \pi_1 X'$ where $X'$ is the mapping torus of a fully irreducible nonsurjective map of a graph and $X'$ is hyperbolic relative to $X$. By \cite{MR2949815}, this implies $\pi_1 X'$ is hyperbolic since it contains no $\BS\left(1,m\right)$ and so \eqref{pr:hyperbolicity} holds.
Since \eqref{pr:negative immersions}$\Rightarrow$\eqref{pr:base is QC}, we have \eqref{pr:base is QC}+\eqref{pr:hyperbolicity} $\Rightarrow$ \eqref{pr:The group has a quasiconvex hierarchy} since $\pi_1X$ splits along $\mathcal{H}$ and the vertex-group is free.

We end this section by stating the following conjectures:
\begin{conje}\label{conje:NG give QC}
\eqref{pr:negative immersions} $\Rightarrow$ \eqref{pr:localQC} and hence \eqref{pr:negative immersions} $\Rightarrow$ \eqref{pr:fgip relative to H}.
\end{conje}

\begin{conje}\label{conje:negative immersions are virtually mapping tori}
If $X$ is a compact $2$-complex with negative immersions, then $\pi_1X$ has a finite index subgroup that is isomorphic to the fundamental group of a mapping torus of an immersion $\psi:H\rightarrow F$ of finite directed height.
\end{conje}

\begin{conje}\label{conje:QC is virtually mapping tori}
If $\mathcal{G}$ is a locally quasiconvex hyperbolic group, then $\mathcal{G}$ has a finite index subgroup that is isomorphic to the fundamental group of a mapping torus of an immersion $\psi:H\rightarrow F$ of finite directed height.
\end{conje}
The motivation for Conjecture~\ref{conje:NG give QC} is the fact that the negative immersions property is a strengthening of \textit{negative sectional curvature} which was shown to imply local quasiconvexity in  \cite{WiseSectional02}. Conjecture~\ref{conje:negative immersions are virtually mapping tori} is motivated by the fact that negative immersions implies a quasiconvex hierarchy which gives virtual specialness. The hope is to show that $X$ is virtually $F_\infty$-by-cyclic. Using a theorem of Feighn-Handel \cite{FeighnHandelCoherence}, we can then show that $\pi_1X$ is isomorphic to a partial mapping torus of a graph immersion. Local quasiconvexity and our theorem show that it must have finite directed height. Conjecture~\ref{conje:QC is virtually mapping tori} is motivated by the lack of counter examples. That is, there are currently no examples of locally quasiconvex hyperbolic groups that don't have negative immersions. It is exciting to believe that there is a characterization using the partial HNN framework.

\bibliographystyle{alpha}
\bibliography{wise.bib}

\end{document}